\documentclass[12pt,reqno]{amsart}

\author{Paul Pollack}
\address{Department of Mathematics\\ University of Georgia\\ Athens, GA 30602}

\email{pollack@uga.edu}


\usepackage{mathtools}
\usepackage{multirow}
\usepackage{empheq}
\usepackage{colonequals}

\usepackage{amsmath,amssymb,amsthm,geometry,mathscinet,enumitem,url,amsrefs}
\usepackage{parskip}
\usepackage{stackengine}
\usepackage{xcolor}
\usepackage{booktabs}

\renewcommand\epsilon\varepsilon
\renewcommand\supset\supseteq
\usepackage[utf8]{inputenc}
\makeatletter
\NewCommandCopy\@@pmod\pmod
\DeclareRobustCommand{\pmod}{\@ifstar\@pmods\@@pmod}
\def\@pmods#1{\mkern4mu({\operator@font mod}\mkern 6mu#1)}
\makeatother

\numberwithin{equation}{section}

\DeclareMathAlphabet{\curly}{U}{rsfs}{m}{n}
\newtheorem{thm}{Theorem}[section]
\newtheorem{cor}[thm]{Corollary}
\newtheorem{prop}[thm]{Proposition}
\newtheorem{lem}[thm]{Lemma}

\newtheorem*{hypp}{Hypothesis U}

\theoremstyle{remark}
\newtheorem*{rmk}{Remark}

\def\Fact{\curly{F}}

\def\sym{\mathrm{sym}}

\def\Ii{\mathcal{I}}
\def\Ll{\curly{L}}

\def\Nn{\mathcal{N}}

\def\Dd{\mathcal{D}}

\newcommand\rad{\mathrm{rad}}

\def\Z{\mathbb{Z}}

\def\Mm{\curly{M}}

\def\Pp{\mathcal{P}}

\def\lcm{\mathop{\mathrm{lcm}}}
\renewcommand\subset\subseteq

\setlist[enumerate,1]{
  label=\textup{(\roman*)},
  ref=\roman*,
  font=\normalfont
}
\begin{document}

\title{Small values of Carmichael's $\lambda$-function}

\begin{abstract} Let $\lambda(n)$ be the exponent of the multiplicative group $(\Z/n\Z)^{\times}$, and set $L(x,y) = \#\{n\le x: \lambda(n)\le y\}$. We prove an upper bound for $\log \frac{L(x,y)}{x}$ valid for $\exp((\log_2{x})^{1+\epsilon}) \le y\le x/\exp((\log_2{x})^{1+\epsilon})$. Our bound is asymptotically sharp under a plausible hypothesis on powersmooth shifted primes. As an application, we obtain a new upper bound on the count of odd $n\le x$ for which the order of $2$ modulo $n$ is appreciably smaller than $x^{1/2}$.
\end{abstract}

\maketitle

\section{Introduction}
Every student passing through a course in elementary number theory learns about Euler's theorem, according to which $a^{\phi(n)}\equiv 1\pmod{n}$ for all integers $a$ with $\gcd(a,n)=1$. Let $\lambda(n)$ denote the least positive integer that can take the place of $\phi(n)$ in this result --- in other words, the exponent of the multiplicative group $(\Z/n\Z)^{\times}$. It was  already known to Gauss that $\lambda(n) = \lcm_{p^k\parallel n}[\lambda(p^k)]$, where $\lambda(p^k) = p^{k-1}(p-1)$ if $p$ is odd or $k \le 2$, and $\lambda(2^k) = 2^{k-2}$ for every $k\ge 3$ (see Articles 82--92 of \textit{Disquisitiones Arithmeticae} \cite{gauss66}). But R.\,D. Carmichael \cite{carmichael10}, working a century later, was the first to study properties of $\lambda(n)$ as an \emph{arithmetic function}; accordingly, one typically refers to $\lambda(n)$ as \textsf{Carmichael's lambda-function}. 

Determining the maximal order of $\lambda(n)$ is trivial: Clearly, $\lambda(n)\le \phi(n) \le n-1$ for all integers $n>1$, with $\lambda(n)=n-1$ whenever $n$ is prime. The average, minimal, and typical (normal) orders are considerably more interesting; these were investigated by Erd\H{o}s, Pomerance, and Schmutz in \cite{EPS91}. Concerning the typical size of $\lambda(n)$, those authors prove the existence of an explicit constant $A \approx 0.227$ for which \begin{equation}\label{eq:lambdamean} \lambda(n) = n (\log{n})^{-\log_3 n- A + o(1)} \end{equation}
as $n$ tends to infinity through a set of asymptotic density $1$. (Here and below, $\log_{k}$ is the $k$th iterate of the natural logarithm.) As for the minimal order of $\lambda(n)$, they show that for a certain constant $C>0$ and  infinitely many $n$,
\[ \lambda(n) \le (\log{n})^{C\log_3 n}, \]
while on the other hand, 
\begin{equation}\label{eq:lambdalower} \lambda(n) \ge (\log{n})^{(\frac{1}{\log{2}} + o(1))\log_3{n}}  \end{equation}
whenever $n$ tends to infinity.

In this paper, we study the distribution of ``intermediate'' values of $\lambda(n)$, meaning values of $\lambda(n)$
lying between the scales suggested by its typical and minimal orders.  Let \[ L(x,y) = \#\{n \le x: \lambda(n) \le y\}. \]
As long as $y$ and $x/y$ are both ``reasonably large'', we obtain an upper bound for $\log \frac{L(x,y)}{x}$ that we expect to be asymptotically sharp. To state our results precisely, it is convenient to let
\[ \Ll(x,y) = \exp\left(\frac{\log{x}\log_3{x}}{\log_2{y}}\right), \quad\text{setting}\quad \Ll(x) := \Ll(x,x) = \exp\left(\frac{\log{x}\log_3{x}}{\log_2{x}}\right). \]

\begin{thm}\label{thm:main} Fix $\epsilon > 0$. Suppose that in the domain $x, y \ge 100$, we have both
\begin{equation}\label{eq:nottoosmall} y \ge \exp((\log_2 x)^{1+\epsilon}) \end{equation}
and
\begin{equation}\label{eq:opaque} y \le x/\exp((\log_2{x})^{1+\epsilon}). \end{equation}
Then, as $x\to\infty$,
\[ \frac{1}{x} L(x,y) \le 
\begin{cases}\Ll(x/y)^{-1+o(1)} &\text{if $y\ge \Ll(x)^2$}, \\
\Ll(x,y)^{-1+o(1)} &\text{otherwise}.
\end{cases}\]
Furthermore, under a plausible hypothesis on powersmooth shifted primes {\rm (\textsf{Hypothesis U} \emph{below})}, equality holds in both cases.
\end{thm}

Some examples of Theorem \ref{thm:main} are recorded in Table \ref{tbl:thmapplications}.

\begin{table}[t]
\centering
\renewcommand{\arraystretch}{2.2} 
\begin{tabular}{l@{\qquad}l@{\qquad}l}
\textbf{Constraint} & \textbf{$y = y(x)$} & \textbf{Upper bound for $L(x,y)$, as $x \to \infty$} \\ \toprule

\multirow{3}{*}{\shortstack{$\alpha \in K \subseteq (0,1)$ \\ \small (compact subset $K$)}} 
    & $x^{1-\alpha}$ & $x \exp \left( -(\alpha + o(1)) \frac{\log x \log_3 x}{\log_2 x} \right)$ \\
    & $\exp((\log x)^\alpha)$ & $x \exp \left( -(\frac{1}{\alpha} + o(1)) \frac{\log x \log_3 x}{\log_2 x} \right)$ \\
    & $x / \exp((\log x)^\alpha)$ & $x \exp \left( -(\frac{1}{\alpha} + o(1)) \frac{(\log x)^\alpha \log_3 x}{\log_2 x} \right)$ \\ \midrule

\multirow{1}{*}{\shortstack{$\beta \in K \subseteq (1,\infty)$ \\ \small (compact subset $K$)}} 
    & $\exp((\log_2 x)^\beta)$ & $x^{1-\frac{1}{\beta}+o(1)}$ \\ \bottomrule
\end{tabular}

\vspace{6pt} 
\caption{A few instances of Theorem \ref{thm:main}.}
\label{tbl:thmapplications}
\end{table}


While $L(x,y)$ does not seem to have been studied previously in the literature, a closely related quantity appears in Theorem 5 of \cite{FPS01}, by Friedlander, Pomerance, and Shparlinski. It is shown there that for all large $x$ and $y$ with $y\ge \exp((\log_2{x})^3)$, 
\begin{equation}\label{eq:FPSbound} \#\{n \le x: \lambda(n) \le ny^{-1}\} \le x\exp(-0.69(\log{y}\log\log{y})^{1/3}). \end{equation}
We can improve this by our methods. Note that the left-hand side of \eqref{eq:FPSbound} is majorized by $L(x,xy^{-1})$. We show in Proposition \ref{prop:largeyupper} that $L(x,xy^{-1}) \le x\Ll(y)^{-1+o(1)}$ whenever $x\to\infty$ and $\exp((\log_2{x})^{1+\epsilon}) \le y \le x/100$.

The quoted normal order theorem of Erd\H{o}s--Pomerance--Schmutz implies that asymptotically 100\% of $n\le x$ have $\lambda(n) < x \exp(-(1+o(1))\log_2{x} \log_3{x})$. So one cannot hope to obtain a nontrivial upper bound on $L(x,y)$ unless $y$ is a bit smaller than $\approx x/\exp(\log_2 x\log_3 x)$. This makes our assumption \eqref{eq:opaque} seem fairly natural. Similarly, \eqref{eq:lambdalower} shows that the only way one can have $\lambda(n) \le \exp(\log_2 x \log_3 x)$ is if $n$ itself is very small (e.g., we must have $n < \exp((\log{x})^{0.7})$). This motivates assuming a condition such as \eqref{eq:nottoosmall}.

We now explain the hypothesis under which the upper bounds in Theorem \ref{thm:main} are sharp. Let $P^{+}(n)$ (respectively, $P^{\ast}(n)$) denote the largest prime (resp., prime power) dividing $n$, setting $P^{+}(1) = P^{\ast}(1) = 1$. We say $n$ is \textsf{$y$-smooth} (resp., \textsf{$y$-powersmooth}) if $P^{+}(n)\le y$ (resp., $P^{\ast}(n)\le y$). The count of $y$-smooth (resp., $y$-powersmooth) $n\le x$ is denoted $\Psi(x,y)$ (resp., $\Psi^{\ast}(x,y)$).

Fix $\epsilon > 0$. It is known that whenever $x,y$, and $u:=\frac{\log{x}}{\log{y}}$ tend to infinity, with $y \ge (\log{x})^{1+\epsilon}$,  
\[ x\exp(-(1+o(1))u \log{u}) \le \Psi^{\ast}(x,y) \le \Psi(x,y) \le  x\exp(-(1+o(1))u \log{u}).\]
That $\Psi(x,y) = x\exp(-(1+o(1))u\log{u})$ throughout this range is a consequence of work of de Bruijn \cites{dB51,dB66} and Erd\H{o}s--Canfield--Pomerance \cite{CEP83}. To deduce the claimed lower bound on $\Psi^{\ast}(x,y)$, one can then appeal to work of Naimi \cite{naimi88}, which gives that the count of \emph{squarefree} smooth $n\le x$ is within a factor of $\exp(o(u\log{u}))$ of $\Psi(x,y)$ in this regime of $x$ and $y$.

Let $\Pi^{\ast}(x,y)$ denote the corresponding count of powersmooth shifted primes $p-1$; that is, 
\[ \Pi^{\ast}(x,y) := \#\{p\le x: P^{\ast}(p-1)\le y\}. \]
The heuristic principle that shifted primes behave like random integers of the same size (up to local considerations) suggests that $\Psi^{\ast}(x,y)/x \approx \Pi^{\ast}(x,y)/\pi(x)$ in a wide range of $x$ and $y$. We will prove lower bounds matching the upper bounds of Theorem \ref{thm:main} conditional on the following precise form of this hypothesis.

\begin{hypp} Fix $\epsilon > 0$. Whenever $x, y$, and $u:=\frac{\log{x}}{\log{y}}$ tend to infinity with $y \ge (\log{x})^{1+\epsilon}$, we have
\[ \Pi^{\ast}(x,y) = \pi(x) \exp(-(1+o(1))u \log{u}). \]
\end{hypp}

(The ``$U$'' is taken from ``\textsf{ultrafriable},'' a term sometimes preferred over ``powersmooth''.) Conjectures of a similar sort were introduced by Pomerance (e.g., \cites{Pom80, PSW80, pomerance81, pomerance89}) to study pseudoprimes and the value-distribution of Euler's totient function; see also \cites{BFPS04,pollack19,pollack19power}. The upper bound implicit in Hypothesis U is known (it follows, for instance, from the results quoted above on $\Psi(x,y)$ together with Theorem 1 of \cite{PS02}). Unfortunately, our applications require \emph{lower} bounds on $\Pi^{\ast}(x,y)$.

By definition, $\lambda(n)$ is the maximal possible order of an element of $(\Z/n\Z)^{\times}$. Thus, for every odd integer $n$, the order $l(n)$ (say) of $2$ modulo $n$ divides $\lambda(n)$. It is natural to ask whether our results have any implications for the distribution of small values of $l(n)$.\footnote{With minor changes, everything we do for $l(n)$ carries over with $2$ replaced by an arbitrary fixed $b \notin \{0,\pm 1\}$, where of course we now impose the condition $\gcd(n,b)=1$.}

It is known that all but $o(x)$ odd integers $n\le x$ satisfy $l(n) > x^{1/2}$. Furthermore, $\frac{1}{2}$ marks the boundary of our current knowledge: We do not know the same result if the exponent $\frac{1}{2}$ is replaced by a larger constant. (But see Theorem 1 of \cite{KP05} for an improvement of $\frac{1}{2}$ by ``an arbitrary $o(1)$,'' as well as a discussion of what is known under the Generalized Riemann Hypothesis.)  The following corollary to Theorem \ref{thm:main} gives a strong upper bound on solutions to $l(n) \le x^{\beta}$ once $\beta$ dips appreciably below $\frac{1}{2}$.

\begin{cor}\label{cor:orders} Fix $\delta > 0$. For all $x$ sufficiently large in terms of $\delta$,
\[ \#\{\text{odd }n\le x: l(n) \le x^{\beta}\} \le x/\Ll(x)^{\frac{1}{2}-\beta-\delta} \]
uniformly in $\beta \ge 0$.
\end{cor}

Corollary \ref{cor:orders} becomes trivial when $\beta \ge \frac{1}{2}$. At the endpoint $\beta=\frac{1}{2}$, we have no nontrivial estimates of comparable strength. In fact, at present we cannot prove that there are $O(x/(\log{x})^2)$  \emph{primes} $p\le x$ with $l(p) \le x^{1/2}$ --- even permitting ourselves the luxury of assuming the Generalized Riemann Hypothesis! 

Notwithstanding our current ignorance, it seems likely the ``correct'' version of Corollary \ref{cor:orders} has $1-\beta$ in place of $\frac{1}{2}-\beta$ in the exponent of $\Ll(x)$. (See the remark concluding \S\ref{sec:corollary}.) Conditional on Hypothesis U, this would be best possible in the
following sense: small modifications to the proof of
Theorem \ref{thm:main} show that, for fixed $\beta\in(0,1)$ and $x\to\infty$, there are at least
$x/\Ll(x)^{1-\beta+o(1)}$ odd integers $n\le x$ with
$\lambda(n)\le x^\beta$. All such $n$ have $l(n)\le x^\beta$
automatically.

\subsection*{Notation} Most of our notation is familiar. For instance, $\omega(n) = \sum_{p\mid n} 1$ and $\Omega(n) = \sum_{p^k \parallel n} k$, where here and below, $p$ always denotes a prime. Perhaps a little less standard are $P^{+}(n)$ for the largest prime factor of $n$ (with $P^{+}(1)=1$) and $P^{-}(n)$ for the least prime factor of $n$ (with $P^{-}(1)=\infty$). We write $\rad(n)$ for the \textsf{radical} of $n$, meaning the largest squarefree divisor of $n$. 


\section{Overview}
We will give separate arguments for the rigorous upper bounds and conditional lower bounds, in each of the two ranges of $y$:
\begin{itemize}
\item (Upper bound for small $y$) In \S\ref{sec:smallyupper}, we establish that $L(x,y) \le x \Ll(x,y)^{-1+o(1)}$ when $\exp((\log_2{x})^{1+\epsilon}) \le y \le x^{1/\log_2{x}}$ and $x\to\infty$ (Proposition \ref{prop:smally}). 
\item (Lower bound for small $y$) In \S\ref{sec:smallylower}, we assume Hypothesis U and show, in Proposition \ref{prop:smallylower}, that $L(x,y) \ge x \Ll(x,y)^{-1+o(1)}$ whenever $x\ge y \ge \exp((\log_2{x})^{1+\epsilon})$ and $x\to\infty$. 
\item (Lower bound for large $y$) In \S\ref{sec:largeylower}, we assume Hypothesis U and show, in Proposition \ref{prop:lowerlarge}, that $L(x,y) \ge x \Ll(x/y)^{-1+o(1)}$ when $y \ge \Ll(x)^2$ and $x/y\to\infty$. 
\item (Upper bound for large $y$) In \S\S\ref{sec:largeyupper}--\ref{sec:largeyupper2}, we show that $L(x,y) \le x\Ll(x/y)^{-1+o(1)}$ whenever $x\to\infty$ and $100 \le y \le x/\exp((\log_2{x})^{1+\epsilon})$ (Proposition \ref{prop:largeyupper}).
\end{itemize}

The attentive reader will notice that together these results \emph{almost} cover all of the assertions of Theorem \ref{thm:main}. However, if  $x^{1/\log_{2}{x}} < y < \Ll(x)^2$, we are claiming in Theorem \ref{thm:main} that $L(x,y) \le x \Ll(x,y)^{-1+o(1)}$, while this range is not included in Proposition \ref{prop:smally}. In fact, this is no problem at all: For these values of $y$, we have $\log \Ll(x,y) \sim \log \Ll(x) \sim \log \Ll(x/y)$, and so our ``upper bound for large $y$'' (Proposition \ref{prop:largeyupper}) furnishes the desired upper bound.

The first three bulleted components will come to us by adapting arguments of Erd\H{o}s and Pomerance, originally applied to study (related, but different) questions concerning the $\lambda$ and $\phi$-functions. 

The chief novelty of the paper is the proof of the upper bound for large $y$. (The cases when $x/y = \exp((\log_2{x})^{O(1)})$ require particularly delicate arguments.) This requires revisiting and reworking methods from \cite{pollack19} and its sequel paper \cite{pollack21}, used to investigate popular subsets for Euler's $\phi$-function and large values of the factorization counting function. Preliminary results are collected in \S\ref{sec:largeyupper}, after which the proof of Proposition \ref{prop:largeyupper} is detailed in \S\ref{sec:largeyupper2}. It seems plausible that some of the new results of \S\ref{sec:largeyupper} will find other applications.

Corollary \ref{cor:orders} is proved in \S\ref{sec:corollary}. The strategy there is inspired by the proof of Theorem 1 in \cite{KP05}.

\section{Upper bound for small $y$}\label{sec:smallyupper}
The following consequence of the prime number theorem appears, in stronger form, as \cite{MV07}*{Lemma 7.4}.
    
\begin{lem}\label{lem:primesum} Let $0 < c < 1$, and let $T$ be a positive real number for which $T^{1-c} \ge \exp(4)$. Then 
\[ \sum_{p \le T} \frac{1}{p^c} \ll \frac{T^{1-c}}{\log(T^{1-c})} + \log \frac{1}{1-c},\]
where the implied constant is absolute.
\end{lem}

\begin{prop}\label{prop:smally} Fix $\epsilon > 0$. If $x,y\to\infty$ with
        \[ \exp((\log_2{x})^{1+\epsilon}) \le y \le x^{1/\log_2{x}}, \]
then 
\begin{equation}\label{eq:rankinupper} L(x,y) \le x\exp\left(-(1+o(1)) \frac{\log{x} \log_3{x}}{\log_2{y}}\right).\end{equation}
\end{prop}

\begin{proof} We give a pointwise upper bound. Put $\ell(x,m) = \#\{n\le x: \lambda(n)=m\}$, so that $L(x,y) = \sum_{m \le y} \ell(x,m)$. We will show that $\ell(x,m)$ is bounded above by the expression on the right of \eqref{eq:rankinupper}, uniformly for positive integers $m\le y$. This suffices to prove the proposition. Indeed, 
\[ \log{y} \ll \frac{\log{x}}{\log_2{x}}, \quad\text{while}\quad \frac{\log{x} \log_3{x}}{\log_2{y}} \gg \frac{\log{x} \log_3{x}}{\log_2{x}}. \] Hence, $y = \exp(o(\log x\log_3 x/\log_2 y))$, and
\begin{align*} L(x,y) &\le y \max_{m \le y} \ell(x,m) \\
&=x\exp\left(-(1+o(1)) \frac{\log{x} \log_3{x}}{\log_2{y}}\right),
\end{align*}
as desired.

To bound $\ell(x,m)$, we lean on an idea of Pomerance (see the proof of \cite{pomerance89}*{Lemma 5.2}). If $\lambda(n)=m$, then $p-1\mid m$ for each prime $p\mid n$. Hence, for every $c>0$, we have (Rankin's trick)
\[ \ell(x,m) \le x^{c} \sum_{\substack{n \le x \\ p\mid n \Rightarrow p-1\mid m}} n^{-c} \le x^c \prod_{p:\, p-1\mid m}\left(1 + \frac{1}{p^c} + \frac{1}{p^{2c}} + \dots\right). \]

Below, we will select a value of $c\gg 1$. (Here we mean that $c$ is bounded below by a positive constant depending only on $\epsilon$.) Assuming this lower bound on $c$ for the time being, $1/p^c + 1/p^{2c} +\dots \le p^{-c} (1-1/2^c)^{-1} \ll p^{-c}$, so that
\[ \prod_{p:\, p-1\mid m}\left(1 + \frac{1}{p^c} + \frac{1}{p^{2c}} + \dots\right) \le \exp\left(\sum_{p:\, p-1\mid m}\frac{1}{p^c}+\frac{1}{p^{2c}}+\dots \right) \le \exp\left(O\left(\sum_{p:\, p-1\mid m}\frac{1}{p^{c}}\right)\right). \]
Furthermore,
\[ \sum_{p:\, p-1\mid m}\frac{1}{p^{c}} \le \sum_{p:\, p-1\mid m}\frac{1}{(p-1)^{c}} \le \sum_{d\mid m} \frac{1}{d^c} \le \prod_{p \mid m}\left(1+\frac{1}{p^c}+\frac{1}{p^{2c}}+\dots\right) \le \exp\left(O\left(\sum_{p\mid m}\frac{1}{p^c}\right)\right). \]
Collecting estimates, we see that for certain constants $C_1, C_2 > 0$ (depending only on $\epsilon$),
\begin{equation}\label{eq:lxmbound}
\ell(x,m) \le x^c \exp\left(C_1 \exp\left(C_2 \sum_{p \mid m} \frac{1}{p^c}\right)\right). \end{equation}

The sum over primes $p$ dividing $m$ is bounded above by the corresponding sum over the first $\omega(m)$ primes, which (as $m\le y$) is in turn bounded above by the sum over the primes not exceeding $2\log{y}$ (once $y$ is large enough). This can be estimated by Lemma \ref{lem:primesum}. We choose $c$ so that $(2\log{y})^{1-c}= \log_2{x}$, i.e.,
\[ c = 1-\frac{\log_3{x}}{\log(2\log{y})}. \]
Then 
\[ \sum_{p \mid m} \frac{1}{p^c} \ll \frac{(2\log{y})^{1-c}}{\log((2\log{y})^{1-c})} + \log \frac{1}{1-c} \ll \frac{\log_2{x}}{\log_3{x}}. \]

Recalling that $y \ge \exp((\log_2{x})^{1+\epsilon})$, we have that
$c > \epsilon (1+\epsilon)^{-1}$, vindicating our previous assumption that $c\gg 1$. Plugging this choice of $c$ into \eqref{eq:lxmbound} gives that
\begin{align*} \ell(x,m) &\le x \exp\left(-\frac{\log{x} \log_3{x}}{\log(2\log{y})}\right) \exp\left(O((\log{x})^{o(1)})\right) \\
&= x \exp\left(-(1+o(1)) \frac{\log{x}\log_3{x}}{\log_2{y}}\right),  \end{align*}
as needed. \end{proof} 

\section{Lower bound for small $y$}\label{sec:smallylower}

\begin{prop}\label{prop:smallylower} Assume Hypothesis U. Fix $\epsilon > 0$. If $x,y\to\infty$ with 
\[ x\ge y \ge \exp((\log_2{x})^{1+\epsilon}), \]
then 
\begin{equation}\label{eq:Lxylower} L(x,y) \ge x\exp\left(-(1+o(1)) \frac{\log{x} \log_3{x}}{\log_2{y}}\right).\end{equation}
\end{prop}

\begin{rmk}\label{rmk:smoothexamples} It will emerge in the proof that the lower bound \eqref{eq:Lxylower} holds even if we restrict $L(x,y)$ to include only $n$ all of whose prime factors are bounded above by $\exp((\log_2{x})^2)$. This observation will be useful momentarily.
\end{rmk}

\begin{proof} We can (and will) assume that $0 < \epsilon < 1$. We follow the arguments of Erd\H{o}s \cite{erdos35} and Pomerance \cites{Pom80,pomerance89} used to construct ``popular'' values of Euler's $\phi$-function. Put
\[ X:= \exp(\log_2 x \log_2 y), \quad Y:= \frac{1}{2}\log{y}, \quad k:=\left\lfloor \frac{\log{x}}{\log{X}}\right\rfloor. \]
We let $S$ be the set of squarefree numbers $n$ constructed as products of $k$ distinct primes from the set
\[ \Pp:= \{p \le X: P^{\ast}(p-1) \le Y\}. \]
If $n \in S$, then $n \le X^{k} \le x$. Furthermore, for large $x,y$,
\[ \lambda(n) = \lcm_{p\mid n}[p-1] \le \lcm[1,2,\dots,\lfloor Y\rfloor] < \exp(2Y) = y. \]
Therefore, it suffices to bound $\#S$ from below by the expression on the right-hand side of \eqref{eq:Lxylower}. 

Putting $U:= \frac{\log{X}}{\log{Y}}$, we have that $U = (1+o(1))\log_2{x}$. Moreover, for large $x$ and $y$, we have $Y/(\log{Y})^2 =\frac{1}{2} \log{y}/(\log(\frac{1}{2} \log{y}))^2 > (\log_2{x})^{1+\frac{1}{2}\epsilon}$, so that
\[ Y > (\log_2{x})^{1+\frac{1}{2}\epsilon} (\log{Y})^2 > (\log_2 x \log_2 y)^{1+\frac{1}{2}\epsilon} = (\log{X})^{1+\frac{1}{2}\epsilon}.\]
Therefore, we may apply Hypothesis U to deduce that
\begin{align}\notag \#\Pp &\label{eq:Pprecise}= \pi(X) \exp\left(-(1+o(1)) U\log{U}\right) \\
&= X \exp(-(1+o(1)) \log_2 x \log_3 x).\end{align}

We proceed to estimate $\#S = \binom{\#\Pp}{k}$. By assumption, $\log_2 y \ge (1+\epsilon)\log_3{x}$, 
so that $X \ge \exp((1+\epsilon)\log_2 x\log_3 x)$. Thus, for large $x,y$,
\[ \#\Pp > \exp\left(\frac{\epsilon}{2} \log_2 x\log_3 x\right) > \log{x} > k, \]
and so in estimating $\#S$ we may invoke the inequality $\binom{A}{B} \ge \left(\frac{A}{B}\right)^{B}$, valid for all pairs of integers $A\ge B\ge 1$.
Referring back to the definition of $k$, we see that
\[ k^k = \exp\left(k\log{k}\right) \le \exp\left(\frac{\log{x}}{\log{X}} \log_2{x}\right) = \exp\left(\frac{\log{x}}{\log_2 y}\right) = \exp\left(o\left(\frac{\log{x} \log_3{x}}{\log_2{y}}\right)\right). \]
Furthermore, the estimate \eqref{eq:Pprecise} for $\#\Pp$ implies that
\begin{align*} (\#\Pp)^{k} &= X^k \exp(-(1+o(1)) k\log_2 x\log_3 x) \\
&\ge \frac{x}{X}\exp\left(-(1+o(1)) \frac{\log{x} \log_3{x}}{\log_2{y}}\right) \\
&= x \exp\left(-(1+o(1)) \frac{\log{x} \log_3{x}}{\log_2{y}}\right). \end{align*}
We conclude that $\#S \ge (\#\Pp)^k/k^k = x \exp(-(1+o(1)) \log x \log_3{x}/\log_2 y)$, as claimed.
\end{proof}

\section{Lower bound for large $y$}\label{sec:largeylower}

\begin{prop}\label{prop:lowerlarge} Assume Hypothesis U.  Whenever $x, y\to\infty$ with 
\[ y \ge \Ll(x)^2\quad\text{and} \quad \frac{x}{y}\to\infty,\] we have
\[ L(x,y) \ge x/\Ll(x/y)^{1+o(1)}. \]
\end{prop}

(Here $o(1)$ means a quantity tending to $0$ as $x/y\to\infty$, under the assumption that $y\ge \Ll(x)^2$.)

\begin{proof} Below, we say that a statement holds ``eventually'' if it holds whenever $x/y$ is sufficiently large, subject to $y\ge \Ll(x)^2$ (and $x,y\ge 100$). Put
\[ A:= \frac{y}{\Ll(x/y)},\quad B:= \frac{x}{y} \Ll(x/y). \]
Let $S$ be the collection of all ordered pairs $(a,b)$,
where 
\begin{equation}\label{eq:Sacond} a\le A \quad\text{with}\quad P^{-}(a) > \exp((\log_2{B})^2), \end{equation}
and 
\begin{equation}\label{eq:Sbcond} b\le B\quad\text{with}\quad P^{+}(b)\le \exp((\log_2{B})^2) \text{ and } \lambda(b) \le \Ll(x/y). \end{equation}
If $n=ab$ for some $(a, b) \in S$, then $n \le x$ and
\[ \lambda(n) = \lambda(ab) \le a \lambda(b) \le A \cdot \Ll(x/y)= y. \]
The restrictions on $P^{-}(a)$ and $P^{+}(b)$ guarantee that the map $(a,b)\mapsto ab$ is one-to-one. So the proposition will be proved if we show that $\#S \ge x/\Ll(x/y)^{1+o(1)}$.

Our assumption that $y\ge \Ll(x)^2$ implies that, eventually,
\[ \frac{y}{\Ll(x/y)} > \Ll(x/y), \quad\text{which in turn yields that}\quad \exp((\log_2{B})^2) = \left(\frac{y}{\Ll(x/y)}\right)^{o(1)}. \]
It now follows from Brun's sieve that the number of $a$ satisfying \eqref{eq:Sacond} is $\gg A/(\log_2{B})^2$. Hence, the number of choices for $a$ is at least $A \cdot \Ll(x/y)^{o(1)}$. On the other hand, by Proposition \ref{prop:smallylower} and the remark following, the number of choices of $b$ is at least $B \cdot \Ll(x/y)^{-1+o(1)}$ (keeping in mind that $B=(x/y)^{1+o(1)}$). Hence, $\#S \ge AB \cdot \Ll(x/y)^{-1+o(1)} = x \Ll(x/y)^{-1+o(1)}$, as desired.
\end{proof}

\section{Upper bound for large $y$: Preparation}\label{sec:largeyupper}
In this section, we collect several preliminary results needed to treat this range of $y$. We begin with a crude form of a result of de Bruijn.

\begin{lem}\label{lem:dB} For each fixed $\delta>0$ and all large $x$, we have $\sum_{n \le x} 1/\rad(n) < x^{\delta}$.
\end{lem}

\begin{rmk} According to Theorem 1 of de Bruijn's paper \cite{dB62}, we in fact have $\sum_{n\le x} 1/\rad(n) = \exp((1+o(1)) \sqrt{\frac{8\log{x}}{\log\log{x}}})$, as $x\to\infty$. \end{rmk}

\begin{proof} Observe that for all $x\ge 1$,
\[ \sum_{n \le x} \rad(n)^{-1} \le x^{\delta/2} \sum_{n \ge 1} \rad(n)^{-1} n^{-\delta/2} = x^{\delta/2} \prod_{p} (1+p^{-1-\frac{1}{2}\delta} + p^{-1-\delta} + p^{-1-\frac{3}{2}\delta} + \dots) \ll x^{\delta/2}. \qedhere\]
\end{proof}

\subsection{Factorizations and compositions} By a \textsf{factorization}, we mean a finite multiset of integers all at least $2$. We denote the factorization consisting of the integers $d_1,\dots,d_k$ as $\langle d_1,\dots,d_k\rangle$, and we say that $\langle d_1,\dots,d_k\rangle$ is a \textsf{factorization of the integer $d_1\cdots d_k$}. We let $f(n)$ denote the number of factorizations of the positive integer $n$ (where, by convention, $f(1)=1$). For example, the factorizations of $18$ are 
\[ \langle 18\rangle, \quad \langle 2, 9\rangle, \quad \langle 3, 6\rangle, \quad \langle 2, 3, 3\rangle, \]
and so $f(18)=4$. (What we call a factorization is sometimes referred to as a \textsf{multiplicative partition}.)

The following proposition is the main result of \cite{pollack21} and appears as that paper's Theorem 1.1. (With the $\phi$-preimage counting function $\#\phi^{-1}(n)$ in place of $f(n)$, this appears earlier as the main theorem of \cite{pollack19}.)

\begin{prop}\label{prop:originalfactprop} Fix $\delta > 0$ and $\beta \in (0,1)$. There is an $x_0 = x_0(\delta,\beta)$ for which the following holds. If $x>x_0$ and $S\subset [1,x]$ is a set of integers with $\#S \le x^{1-\beta}$, then
\begin{equation}\label{eq:fsum} \sum_{n \in S} f(n) \le x/\Ll(x)^{\beta - \delta}. \end{equation}
\end{prop}

We require a variant of Proposition \ref{prop:originalfactprop} with some uniformity in $\beta$, with factorizations replaced by ``multiplicative compositions with not too many parts.''

By a \textsf{multiplicative composition} of $n$, we mean a finite sequence of integers, all at least $2$, whose product is $n$. (This is sometimes called an \textsf{ordered factorization}.) We let $g(n)$ denote the total number of multiplicative compositions of $n$, setting $g(1)=1$, and we let $g_w(n)$ denote the number of multiplicative compositions of $n$ into $w$ parts. (We define $g_0(1) =1$ and $g_0(n)=0$ for $n>1$.) For example, the multiplicative compositions of 18 are 
\[ [18],\quad [2, 9],\quad [9,2],\quad [3,6],\quad [6,3],\quad [2,3,3],\quad [3, 2, 3],\quad [3, 3, 2],\]
so that $g_1(18)=1, g_2(18)=4, g_3(18)=3$, and $g(18)=8$. It will be useful to note that, trivially, $g_w(n) \le \tau_w(n)$ where $\tau_w$ is the $w$-fold divisor function.

Let $x\ge 100$. We say that the integer $w$ \textsf{belongs to the critical range} (with respect to $x$) if
\begin{equation}\tag{K} 1 \le w \le \frac{\log{x}}{(\log_2{x})^2} (\log_3{x})^2. \end{equation}

\begin{prop}\label{prop:orderedfact} Fix $\delta > 0$. There is an $x_0 = x_0(\delta)$ for which the following holds. If $x > x_0$, $\beta \in [\frac{1}{1000},1]$, and $S \subset [1,x]$ is a set of integers with $\#S \le x^{1-\beta}$, then
\[ \sum_{n \in S} g_w(n) \le x/\Ll(x)^{\beta-\delta} \]
for all integers $w$ belonging to the critical range.
\end{prop}
Here $\frac{1}{1000}$ was chosen for definiteness; it may be replaced with any fixed positive real number.

We prove Proposition \ref{prop:orderedfact} by modifying the method used to establish \cite{pollack21}*{Theorem 1.1}. The following two lemmas are analogues of \cite{pollack21}*{Lemma 2.1, Lemma 2.2}. Below,  $\Omega_{>z}(n):= \sum_{p^e \parallel n,~p > z} e$.

\begin{lem}\label{lem:AZ} Let
\[ z = \exp((\log_2{x})^{1/2}).\]
Fix any $\eta \in (0,1)$, and let 
\[ A = (\log_2{x})^{1-\eta}. \]
As $x \to\infty$,
\[ \sum_{n\le x} A^{\Omega_{>z}(n)} g_w(n) \le x \Ll(x)^{o(1)}, \]
where the decay of the $o(1)$ term to $0$ is uniform for $w$ in the critical range.
\end{lem}

\begin{proof} We once again apply Rankin's trick. For any choice of $c > 1$, 
\[ \sum_{n\le x} A^{\Omega_{>z}(n)} g_w(n) \le x^c \sum_{n \ge 1} \frac{A^{\Omega_{>z}(n)} g_w(n)}{n^c} = x^c \left(\sum_{d \ge 2} \frac{A^{\Omega_{>z}(d)}}{d^c}\right)^w.  \]
This exact sum on $d$ appears already in the proof of Lemma 2.1 of \cite{pollack21}, where it is shown that 
\[ \sum_{d \ge 2} \frac{A^{\Omega_{>z}(d)}}{d^c} \ll \log{z}+ \frac{\Gamma(A)}{(c-1)^{A}}, \]
uniformly for $c \in (1,2)$. 

We now select $c$ in order that
\[ c-1 = \frac{Aw}{\log{x}}. \]
Let
\begin{equation}\label{eq:Wupperlimit} W=\frac{\log{x}}{(\log_2{x})^2} (\log_3{x})^2 \end{equation} be the upper limit of the critical range (K). (Note that $1 < c \le 1 +\frac{AW}{\log{x}} < 2$ for all large $x$.) Then 
\[ x^c = x\exp(Aw) \le x\exp(AW) = x\Ll(x)^{o(1)}. \]
Furthermore, for a certain large positive constant $C_4$,
\begin{align*} \left(\sum_{d \ge 2} \frac{A^{\Omega_{>z}(d)}}{d^c}\right)^{w} &\le C_4^{w}\left(\log{z}+ \frac{\Gamma(A)}{(c-1)^{A}}\right)^w  \\
&\le (2C_4)^{w} \left((\log{z})^{w} + \frac{\Gamma(A)^{w}}{(c-1)^{Aw}}\right).
\end{align*}
Here we have
\begin{align*} (2C_4)^{w} &\le \exp(O(W)) = \Ll(x)^{o(1)}, \\(\log{z})^{w} &\le \exp(O(W \log_3 x)) = \Ll(x)^{o(1)},
\\ \Gamma(A)^{w} &\le \exp(O(AW\log{A})) = \Ll(x)^{o(1)}, \end{align*}
and
\begin{align*} (c-1)^{-Aw} &= \left(\frac{\log{x}}{Aw}\right)^{Aw} \le \left(\frac{\log{x}}{AW}\right)^{AW} \\ &\le \exp(O(AW \log_3 x)) = \Ll(x)^{o(1)}.
\end{align*}
(In the first line of the last display, we use that $t\mapsto (\frac{\log{x}}{t})^{t}$ is increasing for $t \le \frac{1}{\mathrm{e}}\log{x}$, while $AW \le \frac{1}{\mathrm{e}}\log{x}$.) Collecting estimates completes the proof.
\end{proof}

\begin{lem}\label{lem:Blem} As $x \to\infty$,
\[ \sum_{n\le x} (3/2)^{\Omega(n)} g_w(n) \le x\Ll(x)^{o(1)}, \]
where the decay of the $o(1)$ term to $0$ is uniform for $w$ in the critical range.
\end{lem}
\begin{proof} Put $S(T):= \sum_{m \le T} (3/2)^{\Omega(m)}$. Then $S(T) \le \sum_{m \le T} 2^{\Omega(m)} \ll T (\log{T})^2$ for all $T \ge 2$ (for the last estimate, see \cite{grosswald56} or \cite{MV07}*{Exercise 15(c), p.\ 42}). Thus, uniformly for  $c \in (1,2)$, 
\[ \sum_{d \ge 2} \frac{(3/2)^{\Omega(d)}}{d^c} \ll \int_{2}^{\infty} \frac{(\log{t})^2}{t^c}\,\mathrm{d}t < \int_{1}^{\infty} \frac{(\log{t})^2}{t^c}\,\mathrm{d}t = \frac{1}{(c-1)^3} \int_{0}^{\infty} s^2 \mathrm{e}^{-s}\,\mathrm{d}s = \frac{2}{(c-1)^3}. \]
(Here we have made the substitution $t = \mathrm{e}^{s/(c-1)}$.) Therefore, for $w$ in the critical range and $c \in (1,2)$, 
\[ \sum_{n \le x} (3/2)^{\Omega(n)} g_w(n) \le x^c \sum_{n\ge 1} \frac{(3/2)^{\Omega(n)} g_w(n)}{n^c} = x^c \left(\sum_{d \ge 2} \frac{(3/2)^{\Omega(d)}}{d^c}\right)^w \le x^{c} C_5^{w} (c-1)^{-3w}, \]
where $C_5$ is a certain large (positive) absolute constant.

Let $W$ be as in \eqref{eq:Wupperlimit}. We apply the upper bound of the last paragraph with $c$ chosen to satisfy $c-1 = \frac{3w}{\log{x}}$. Then $x^c \le x\exp(3W) = x \Ll(x)^{o(1)}$ and $C_5^{w} \le C_5^{W} = \Ll(x)^{o(1)}$.  Furthermore, 
\[ (c-1)^{-3w} = \left(\frac{\log{x}}{3w}\right)^{3w} \le  \left(\frac{\log{x}}{3W}\right)^{3W} \le \exp(O(W \log_3 x)) = \Ll(x)^{o(1)}.  \]
Again, collect estimates.
\end{proof}

We can now give the proof of Proposition \ref{prop:orderedfact}. 

\begin{proof}[Proof of Proposition \ref{prop:orderedfact}] We can and will assume $\delta \in (0,1)$. We begin by fixing $\eta := \frac{1}{4}\delta \in (0,1)$; this is the value of $\eta$ with which we will apply Lemma \ref{lem:AZ}.

Let us consider the contribution to $\sum_{n \in S} g_w(n)$ made by $n$ with $\Omega_{>z}(n) > (1-\eta)\beta \frac{\log{x}}{\log_2{x}}$. By Lemma \ref{lem:AZ}, for all $w$ in the critical range,
\[ A^{(1-\eta)\beta\log{x}/\log_2{x}}\sum_{\substack{n \in S\\\Omega_{>z}(n) > (1-\eta)\beta \frac{\log{x}}{\log_2{x}}}} g_w(n) \le  \sum_{n \le x} A^{\Omega_{>z}(n)} g_w(n) \le x \Ll(x)^{o(1)}. \]
Here and below, $o(1)$ means an expression that decays to $0$ as $x\to\infty$, uniformly in all parameters \emph{except} possibly the fixed quantity $\delta$. Rearranging,
\[ \sum_{\substack{n \in S\\\Omega_{>z}(n) > (1-\eta)\beta \frac{\log{x}}{\log_2{x}}}} g_w(n)  \le x \Ll(x)^{o(1)} \cdot A^{-(1-\eta)\beta\log{x}/\log_2{x}} = x/\Ll(x)^{(1-\eta)^2 \beta + o(1)}. \]
As
\begin{equation}\label{eq:funnyetainequality} (1-\eta)^2 \beta > \beta - 2\eta \beta = \beta - \frac{1}{2}\delta\beta \ge \beta - \frac{1}{2}\delta, \end{equation}
it follows that
\[ \sum_{\substack{n \in S\\\Omega_{>z}(n) > (1-\eta)\beta \frac{\log{x}}{\log_2{x}}}} g_w(n) < \frac{1}{2} x/\Ll(x)^{\beta-\delta} \] for all large enough $x$ (here and below, ``large enough'' means ``with respect to a threshold depending only on the fixed parameter $\delta$'').  

Similarly, Lemma \ref{lem:Blem} implies that
\[ \sum_{\substack{n \in S \\ \Omega(n) > \log{x}/(\log_2{x})^{2/3}}} g_w(n) \le x \Ll(x)^{o(1)} \cdot (2/3)^{\log{x}/(\log_2{x})^{2/3}}, \]
for all $w$ in the critical range. Here the right-hand expression is $o(x/\Ll(x)^{\beta-\delta})$.

Call the positive integer $n\le x$ \textsf{admissible} if 
\[ \Omega_{>z}(n) \le (1-\eta)\beta \frac{\log{x}}{\log_2{x}} \quad\text{and}\quad\Omega(n) \le \frac{\log{x}}{(\log_2{x})^{2/3}}. \]
To finish off the proposition, it is enough to show that (for $x$ large, and $w$ in the critical range) the sum of $g_w(n)$, taken over admissible $n\in S$, is smaller than $\frac{1}{3} x/\Ll(x)^{\beta-\delta}$. In fact, we will prove that this  sum is bounded by the much smaller quantity $x^{1-\frac{1}{2}\beta\eta}$. 

Suppose $n \in S$ is admissible, and write $n = n_0 n_1$, where $n_0$ is $z$-smooth and every prime factor of $n_1$ exceeds $z$. Then $g_w(n) \le \tau_w(n) = \tau_w(n_0) \tau_w(n_1)$. Since $n_0$ is $z$-smooth and $\Omega(n) \le \log{x}/(\log_2 x)^{2/3}$, we have that
\[ n_0 \le z^{\Omega(n)} \le \exp(\log{x}/(\log_2{x})^{1/6}) = x^{o(1)}. \]
Hence (uniformly for $w$ in the critical range),
\[ \tau_w(n_0) \le n_0^2 \sum_{m} \frac{\tau_w(m)}{m^2} = n_0^2 \cdot \zeta(2)^w < 2^w n_0^2 = x^{o(1)}. \]
Furthermore (again for $w$ in the critical range),
\[ \tau_w(n_1) \le w^{\Omega(n_1)} = w^{\Omega_{>z}(n)} \le w^{(1-\eta)\beta \log{x}/\log_2{x}} \le x^{(1-\eta)\beta+o(1)}. \]
Recalling that $\#S \le x^{1-\beta}$, we conclude that 
\[ \sum_{\substack{n \in S \\n \text{ admissible}}} g_w(n) \le x^{(1-\eta)\beta+o(1)} \#S \le x^{1-\beta\eta + o(1)}, \]
which is indeed smaller than $x^{1-\frac{1}{2}\beta\eta}$ for large enough $x$.
\end{proof}
\subsection{Euler's function} Our proof requires precise estimates for the frequency with which $\phi(n)$ possesses a given divisor $d$. The next lemma is a refinement, in terms of multiplicative compositions, of \cite{pollack19}*{Lemma 7} (which was expressed in terms of  factorizations). 

Let $\mathbf{d}=\langle d_1,\dots,d_k\rangle$ be a factorization with distinct parts $d_1',\dots,d_{\ell}'$, appearing with respective multiplicities $m_1,\dots,m_\ell$. The \textsf{symmetry constant} of $\mathbf{d}$, denoted $\sym(\mathbf{d})$, is the product $m_1! \cdots m_{\ell}!$.

\begin{lem}\label{lem:PV2} There is an absolute constant $C_{\phi}>0$ for which the following holds{\rm :} Let $x\ge 3$, and let $W > 0$. For each positive integer $d$, the number of squarefree positive integers $n \le x$ for which
\[ \omega(n)\le W\quad\text{and}\quad d\mid \phi(n) \]
is at most
\[ \frac{x}{d} \sum_{0 \le w \le W} \frac{(C_{\phi}(\log_2{x})(\log_2(3d)))^{w}}{w!} g_w(d). \]
\end{lem}

\begin{proof} We may assume that $1 < d \le x$. Let $\Fact(d)$ denote the set of factorizations of $d$. If $\mathbf{d} \in \Fact(d)$, we say that the positive integer $n$ \textsf{corresponds to $\mathbf{d}$} if, writing $\mathbf{d} = \langle d_1,\dots,d_k\rangle$, there are distinct primes $p_1,\dots,p_k$ dividing $n$ with each $p_i\equiv 1\pmod{d_i}$. 

Suppose $n$ is squarefree and $d\mid \phi(n)$. Write $n = p_1\cdots p_t$, where the $p_i$ are distinct primes and $t=\omega(n)\le W$. Then $d\mid (p_1-1)\cdots (p_t-1)$. Hence, we can write $d=d_1\cdots d_t$, where each $d_i \mid p_i-1$. Reordering, we can assume that $d_i > 1$ precisely for $1 \le i \le w$. This shows: Whenever $d\mid \phi(n)$, the integer $n\le x$ corresponds to some $\mathbf{d} \in \Fact(d)$ of length at most $W$ (namely, $\mathbf{d} = \langle d_1,\dots, d_w\rangle$). 

Now take any $\mathbf{d} \in \Fact(d)$. Suppose $\mathbf{d}$ has length $w$, say $\mathbf{d} = \langle d_1,\dots,d_w\rangle$. Let $d_1', \dots, d_{\ell}'$ be the distinct parts in $\mathbf{d}$, and let $m_1,\dots,m_{\ell}$ be their multiplicities (so that $m_1+\dots+m_\ell=w$). Then the number of $n\le x$ corresponding to $\mathbf{d}$ is bounded above by 
\begin{align*} x\prod_{j=1}^{\ell} \frac{1}{m_j!}\Bigg(\sum_{\substack{p \le x\\ p\equiv 1\pmod{d_j'}}} \frac{1}{p}\Bigg)^{m_j} &\le x \prod_{j=1}^{\ell}\frac{1}{m_j!}\left(\frac{C_{\phi}(\log_2{x}) (\log_2{(3d)})}{d_j'}\right)^{m_j} \\&= \frac{x(C_{\phi}(\log_2 x) (\log_2{(3d)}))^{w}}{d \,\sym(\mathbf{d})}. \end{align*}
(We use here that $\sum_{p \le x,\,p\equiv1\pmod{d_j'}} 1/p \ll \log_2{x}/\phi(d_j') \ll \log_2 x \log_2(3d_j')/d_j'$, from the Brun--Titchmarsh inequality and the minimal order of Euler's function.) Letting $\mathbf{d}$ (resp.\ $\mathbf{D}$) run over all factorizations (resp.\ compositions) of $d$ of length $w$,
\[ \sum_{\mathbf{d}} \frac{1}{\sym(\mathbf{d})} = \frac{1}{w!} \sum_{\mathbf{d}} \frac{w!}{\sym(\mathbf{d})} =\frac{1}{w!} \sum_{\mathbf{D}} 1 = \frac{g_w(d)}{w!}, \]
since each length $w$ factorization $\mathbf{d}$ is induced by $w!/\sym(\mathbf{d})$ compositions $\mathbf{D}$. (``Induced'' means ``recovered by forgetting order.'') Collecting estimates, the number of $n\le x$ corresponding to a length $w$ element of $\Fact(d)$ is bounded above by
\[ \frac{x}{d} g_w(d) \frac{(C_{\phi}(\log_2 x) (\log_2{(3d)}))^{w}}{w!}. \]
The proof of Lemma \ref{lem:PV2} is completed by summing on nonnegative integers $w\le W$.
\end{proof}

\section{Upper bound for large $y$: Completion of the proof}\label{sec:largeyupper2}
\begin{prop}\label{prop:largeyupper} Fix $\epsilon > 0$. Suppose $x\to\infty$ and
\begin{equation}\label{eq:mainconstraint} 100 \le y \le x/\exp((\log_2{x})^{1+\epsilon}). \end{equation}
Then
\[ L(x,y) \le x \Ll(x/y)^{-1+o(1)}. \]
\end{prop}

\begin{proof} We start by reducing to the squarefree case. Suppose that under the same assumptions on $x,y$, the count of \emph{squarefree} $n\le x$ with $\lambda(n)\le y$ is at most $x \Ll(x/y)^{-1+o(1)}$. We take an arbitrary $n\le x$ with $\lambda(n)\le y$ and write $n=n_0 n_1$, where $n_0$ is squarefull, $n_1$ is squarefree, and $\gcd(n_0,n_1)=1$. For the sake of proving Proposition \ref{prop:largeyupper}, we may assume that $n_0 \le \Ll(x/y)^2$, as only $O(x \Ll(x/y)^{-1})$ integers up to $x$ are divisible by a squarefull number exceeding $\Ll(x/y)^2$. For each $n_0$, we count corresponding squarefree values of $n_1 \le x/n_0$. Certainly $\lambda(n_1) \le \lambda(n_0 n_1) \le y$. Under \eqref{eq:mainconstraint}, we have that 
\[ \frac{x}{n_0 y} > \left(\frac{x}{y}\right)^{1/2} > \exp\left(\frac{1}{2} (\log\log{x})^{1+\epsilon}\right) > \exp\left((\log\log{x})^{1+\frac{1}{2}\epsilon}\right) \ge \exp\left( (\log\log{(x/n_0)})^{1+\frac12\epsilon}\right),\]
so that \[ 100 \le y \le \frac{x/n_0}{\exp((\log_2(x/n_0))^{1+\frac{1}{2}\epsilon})}, \] uniformly in our range of $n_0$. Thus, applying our squarefree version, the number of $n_1$ corresponding to a given $n_0$ is at most $x n_0^{-1} \Ll(x/n_0 y)^{-1+o(1)} = x n_0^{-1} \Ll(x/y)^{-1+o(1)}$. Summing on $n_0$ gives the proposition as originally stated.

We may reduce further to the case when $\lambda(n)$ has the same order of magnitude as $y$. That is, it suffices to prove the upper bound of the proposition for
\[ L^{\ast}(x,y) := \#\{\text{squarefree } n\le x: \frac12 y< \lambda(n) \le y\}. \]
Indeed,
\[ \#\{\text{squarefree }n\le x: \lambda(n)\le y\}\le L(x,x^{1/\log_2{x}}) + \sum_{\substack{m \ge 0 \\ 2^{-m} y > x^{1/\log_2{x}}}}  L^{\ast}(x,2^{-m} y), \]
while $L(x,x^{1/\log_2{x}}) \le x\Ll(x)^{-1+o(1)}$ by Proposition \ref{prop:smally}. If we have the $L^{\ast}$-variant of Proposition \ref{prop:largeyupper}, then $L^{\ast}(x,2^{-m}y) \le x \Ll(x/y)^{-1+o(1)}$, uniformly in these $m$. We conclude by noting that there are $\ll \log{y} = \Ll(x/y)^{o(1)}$ values of $m$.

The rest of the proof is focused on proving this theorem about $L^{\ast}(x,y)$. More precisely, we show the following: Fix $\epsilon \in (0,1)$ and $\delta \in (0,\frac{1}{10})$. If $x$ is sufficiently large in terms of $\epsilon, \delta$, and $x,y$ satisfy \eqref{eq:mainconstraint}, then
\[ L^{\ast}(x,y) \le x \Ll(x/y)^{12\delta-1}. \]

Below, all estimates are made under the assumption that $x$ is sufficiently large (and that \eqref{eq:mainconstraint} holds). We emphasize that ``large'' means ``in terms of $\epsilon$ and $\delta$'' \emph{only}; the largeness thresholds are uniform in all other parameters.

We begin by throwing away all $n$ that do \emph{not} satisfy
\begin{equation}\label{eq:omegacond} \omega(n)\le W, \quad\text{where}\quad  W:= \left\lfloor \frac{10}{\epsilon} \frac{\log\frac{x}{y} \log_3 \frac{x}{y}}{(\log_2\frac{x}{y})^2}\right\rfloor.\end{equation}
By a well-known theorem of Hardy and Ramanujan, the number of $n$ we discard in this way is
\[ \ll \frac{x}{\log{x}} \sum_{w > W}\frac{(\log_2{x} + C_3)^{w-1}}{(w-1)!}. \]
(Here $C_3$ is a certain absolute constant.) 
As $\log \frac{x}{y} \ge (\log_2{x})^{1+\epsilon}$, we have  that $W \ge (\log_2{x})^{1+\frac{1}{2}\epsilon}$ and that each term in the right-hand sum on $w$ is at most half the previous. Thus,
\begin{align*} \frac{x}{\log{x}}\sum_{w > W} \frac{(\log_2{x}+C_3)^{w-1}}{(w-1)!} &\le 2\frac{x}{\log{x}} \frac{(\log_2{x}+C_3)^{W}}{W!} \le x
\frac{(2\log_2{x})^{W}}{W!} \le x\left(\frac{2\mathrm{e}\log_2{x}}{W}\right)^{W}.\end{align*}
Continuing, we observe that  $W^{1-\frac{1}{4}\epsilon} \ge 2\mathrm{e}\log_2{x}$, and so 
\[ \left(\frac{2\mathrm{e}\log_2{x}}{W}\right)^{W} = \exp\left(-W \log\frac{W}{2\mathrm{e}\log_2{x}}\right) \le \exp\left(-\frac{1}{4}\epsilon W \log{W}\right) < \Ll(x/y)^{-2}. \]
Thus, the number of $n$ removed in this initial step is smaller than $x\Ll(x/y)^{-1}$. 

Next, we place the integers $n \le x$ into dyadic intervals. For each $j \in \Z^{+}$, let \[ \Nn_j = (2^{-j} x, 2^{1-j} x].\] At the cost of excluding $O(x\Ll(x/y)^{-1})$ values of $n\le x$, we can (and will) assume that $n \in \Nn_j$ where 
\begin{equation}\label{eq:jcond}2^{-j} x >  x\Ll(x/y)^{-1}. \end{equation}

We continue by sorting the ratios $\frac{\phi(n)}{\lambda(n)}$ dyadically. Suppose $n \in \Nn_j$ and $\lambda(n)\in (y/2,y]$. Then 
\[ \frac{\phi(n)}{\lambda(n)} > \frac{\frac12 n/\log_2{n}}{\lambda(n)} > 2^{-1-j} \frac{x}{y\log_2{x}},  \]
while
\[ \frac{\phi(n)}{\lambda(n)} \le \frac{2^{1-j} x}{y/2} = 4 \cdot 2^{-j} \frac{x}{y}. \]
(In the lower bound, we use again the known minimal order of $\phi(n)$.) Put \begin{equation}\label{eq:Djdef} D_j = 2^{-1-j} \frac{x}{y\log_2{x}},\end{equation} and for each positive integer $k$, set
\[ \Dd_{j,k} = (2^{k-1} D_j, 2^{k} D_j]. \] Then $\frac{\phi(n)}{\lambda(n)} \in \Dd_{j,k}$ for a positive integer $k$ with 
\begin{equation}\label{eq:kcond} 2^{k} D_j \le 8 \cdot 2^{-j} \frac{x}{y}. \end{equation}

Our strategy will be to bound, for each $j$ satisfying \eqref{eq:jcond} and $k$ satisfying \eqref{eq:kcond}, the number of squarefree $n \in \Nn_j$ satisfying \eqref{eq:omegacond} for which $\frac{\phi(n)}{\lambda(n)} \in \Dd_{j,k}$. The desired estimate for $L^{\ast}(x,y)$ will follow upon summing on $j$ and $k$. 

Suppose $d=\frac{\phi(n)}{\lambda(n)}$. The integers $\lambda(n)$ and $\phi(n)$ share the same set of prime divisors. (By the theorems of Cauchy and Lagrange, the exponent of a finite group always has the same set of prime divisors as the order of the group.) Hence,
\[ d\,\rad(d) \mid d\,\rad(\phi(n)) = d\,\rad(\lambda(n)) \mid d\lambda(n) = \phi(n). \]
It therefore suffices to bound the number of squarefree $n \in \Nn_j$ satisfying \eqref{eq:omegacond} with $\phi(n)$ divisible by $d\, \rad(d)$ for some integer $d \in \Dd_{j,k}$.

We stratify dyadically one more time. Let $d \in \Dd_{j,k}$. As $d\,\rad(d) \in [d,d^2]$, there is some $\ell \in \Z^{+}$ with
\[ d\,\rad(d) \in \Ii_{j,k,\ell} := (2^{\ell-1}\cdot 2^{k-1}D_j, 2^{\ell} \cdot 2^{k-1}D_j], \]
for a value of $\ell$ with
\begin{equation}\label{eq:ellcond} 2^{\ell} \cdot 2^{k-1} D_j \le 2 \cdot(2^k D_j)^2. \end{equation} 

Letting $\ell$ run over positive integers satisfying \eqref{eq:ellcond}, we have from Lemma \ref{lem:PV2} and the observation that $d\,\rad(d) \le d^2 \le 2^{2k} D_j^2 < 64 (x/y)^2$,
\begin{align} \#\{\text{squarefree }&n \in \Nn_j:~\omega(n)\le W,~d\,\rad(d)\mid \phi(n)\text{ for some }d \in \Dd_{j,k}\} \notag\\ &\le \sum_{\ell} \sum_{\substack{d\in \Dd_{j,k} \\ d\,\rad(d) \in \Ii_{j,k,\ell}}} \#\{n \le 2^{1-j} x: d\,\rad(d)\mid \phi(n)\} \notag\\
&\le 2^{1-j} x \sum_{\ell} \sum_{\substack{d\in \Dd_{j,k} \\ d\,\rad(d) \in \Ii_{j,k,\ell}}} \frac{1}{d\,\rad(d)}\sum_{1 \le w \le W} \left(2C_{\phi}(\log_2{x}) \left(\log_2 \frac{x}{y} \right)\right)^{w} \frac{g_w(d\, \rad(d))}{w!}\label{eq:tosumfact}.
\end{align}
(We start the sum at $w=1$ rather than $w=0$ since $d\,\rad(d) > 1$ for our values of $d$.) 

We focus attention on the right-hand inner double sum on $d, w$. This will be estimated (uniformly in $j,k,\ell$) via Proposition \ref{prop:orderedfact}. To ease notation, when convenient below, we suppress subscripts indicating dependence on $j,k,\ell$.

The map $d \mapsto d\,\rad(d)$ is injective (on the entire domain of positive integers). We let 
\[ S = \{d\, \rad(d): d \in \Dd_{j,k},~d\,\rad(d) \in \Ii_{j,k,\ell}\}. \]
To estimate $\#S$, observe that if $d \in \Dd_{j,k}$ and $d\,\rad(d) \in \Ii_{j,k,\ell}$, then 
\[ \frac{1}{\rad(d)} = \frac{d}{d\,\rad(d)} \ge \frac{2^{k-1} D_j}{2^{\ell} \cdot 2^{k-1} D_j} = \frac{1}{2^{\ell}}. \]
As $\sum_{d \in \Dd_{j,k}} 1/\rad(d) \le \sum_{d \le 2^k D_j} 1/\rad(d) \le (2^k D_j)^{\delta}$, we conclude from the last display that
\[ \#S \le 2^{\ell} (2^{k} D_j)^{\delta}. \]

To put this in a form to which we can apply Proposition \ref{prop:orderedfact}, let $X =2^{\ell} \cdot 2^{k-1} D_j$ and $s = 2^{\ell}  (2^k D_j)^{\delta}$, so that
\[ S \subset [1,X] \quad\text{with}\quad \#S \le s. \]
We define $\beta$ by the relation \[ s = X^{1-\beta}. \]

From \eqref{eq:ellcond}, we have $2^{\ell} \le 4 \cdot 2^{k} D_j$, and thus
\begin{align} X &= 2^{\ell} \cdot 2^{k-1} D_j\notag \\ &\le 2 \cdot (2^k D_j)^2.\label{eq:Xupper} \end{align}
Hence, $2^k D_j \ge \frac{1}{2} X^{1/2}$, and
\begin{align*}
s = 2^{\ell} (2^k D_j)^{\delta} = \frac{X}{2^{k-1} D_j} (2^k D_j)^{\delta} = 2X (2^{k} D_j)^{\delta-1} \le 4 X^{\frac{1}{2}+ \frac{1}{2}\delta} < X^{0.6},
\end{align*}
recalling our restriction that $\delta \in (0,\frac{1}{10})$. It follows that $\beta \in [\frac{2}{5},1]$. 

We now return to estimating the sum over $d,w$ in \eqref{eq:tosumfact}. We claim that all $1\le w \le W$ fall into the critical range for the parameter $X$. 
To see this, we refer back to \eqref{eq:jcond} and \eqref{eq:Djdef}. These give us that
\begin{equation}\label{eq:Xlower} X = 2^{\ell} \cdot 2^{k-1} D_j > D_j = 2^{-j} x \cdot \frac{1}{2 y \log_2{x}}
\ge \frac{1}{2 \log_2{x}} \frac{x/y}{\Ll(x/y)} = (x/y)^{1+o(1)}. \end{equation}
Thus (see the definition \eqref{eq:omegacond} of $W$),
\begin{equation}\label{eq:critical2} W < \frac{11}{\epsilon}\frac{\log{X} \log_3{X}}{(\log_2{X})^2} < \frac{\log{X} (\log_3{X})^2}{(\log_2{X})^2}, \end{equation}
yielding criticality.

Continuing, notice that for all $w\le W$, 
\[ \left(\log_2 \frac{x}{y}\right)^{w} \le \left(\log_2 \frac{x}{y}\right)^{W} \le \exp\left(\frac{10}{\epsilon}\frac{\log\frac{x}{y}(\log_3\frac{x}{y})^2}{(\log_2 \frac{x}{y})^2}\right)= \Ll(x/y)^{\frac{10}{\epsilon}\frac{\log_3 \frac{x}{y}}{\log_2\frac{x}{y}}} = \Ll(x/y)^{o(1)}. \]
Furthermore, as $\Ll(x/y) \ge \Ll(\exp((\log_2{x})^{1+\epsilon})) > \exp((\log_2{x})^{1+\frac{1}{2}\epsilon})$, we have
\[ \frac{(2 C_{\phi}\log_2{x})^{w}}{w!} \le \sum_{w'\ge 0} \frac{(2C_{\phi}\log_2{x})^{w'}}{w'!} = \exp(2 C_{\phi}\log_2{x}) = (\log{x})^{2C_\phi} = \Ll(x/y)^{o(1)}. \]
Therefore, bearing in mind that every element of $\Ii_{j,k,\ell}$ exceeds $\frac{1}{2}X$, 
\begin{align} \sum_{\substack{d\in \Dd_{j,k} \\ d\,\rad(d) \in \Ii_{j,k,\ell}}}\frac{1}{d\,\rad(d)}\sum_{1\le w \le W} &\frac{\left(2 C_{\phi} (\log_2{x}) \left(\log_2 \frac{x}{y}\right)\right)^{w}}{w!} g_w(d\,\rad(d)) \notag\\ &\le 2X^{-1} \Ll(x/y)^{o(1)} \sum_{\substack{d\in \Dd_{j,k} \\ d\,\rad(d) \in \Ii_{j,k,\ell}}}\sum_{1 \le w \le W} g_w(d\,\rad(d)). \label{eq:bearing}
\end{align}

Applying Proposition \ref{prop:orderedfact},
\begin{align*} \sum_{\substack{d\in \Dd_{j,k} \\ d\,\rad(d) \in \Ii_{j,k,\ell}}}\sum_{1\le w \le W} g_w(d\,\rad(d)) &= \sum_{1\le w \le W} \sum_{s \in S} g_w(s) \le W \cdot X/\Ll(X)^{\beta-\delta} \le X/\Ll(X)^{\beta-2\delta}.\end{align*}
Substituting into \eqref{eq:bearing}, we conclude that the double sum on $d,w$ in \eqref{eq:tosumfact} is bounded above by $\Ll(X)^{2\delta-\beta} \Ll(x/y)^{\delta}$ (say).

Let us see how large this last expression is in terms of $\Ll(x/y)$. Notice that
\[ X^{\beta} = \frac{X}{s} = \frac{2^\ell \cdot 2^{k-1} D_j}{2^{\ell} (2^k D_j)^{\delta}} =  \frac{1}{2} (2^k D_j)^{1-\delta} > (x/y)^{1-2\delta}, \]
using in the last step that $D_j \ge (x/y)^{1+o(1)}$, from \eqref{eq:Xlower}. Hence,
\[ \Ll(X)^{\beta} = \Ll(X^{\beta})^{1+o(1)}> \Ll(x/y)^{1-3\delta}. \]
Furthermore, \eqref{eq:kcond} and \eqref{eq:Xupper} yield $X < 128 (x/y)^2$, and so 
\[ \Ll(X) \le \Ll(x/y)^{2+o(1)}. \]
Therefore,
\begin{align*} \Ll(X)^{2\delta-\beta} \Ll(x/y)^{\delta} &= \Ll(X)^{2\delta} \Ll(X)^{-\beta} \Ll(x/y)^{\delta}  \\
&\le \Ll(x/y)^{5\delta} \Ll(x/y)^{3\delta-1} \Ll(x/y)^{\delta} \\
&= \Ll(x/y)^{9\delta-1}.
\end{align*}

Referring back to \eqref{eq:tosumfact},  we deduce that for each $k$ and $j$, 
\begin{equation}\label{eq:finallydone}\#\{
\text{squarefree }n \in \Nn_j: \omega(n)\le W, d\,\rad(d)\mid \phi(n)\text{ for some }d \in \Dd_{j,k}\} \le 2^{1-j} x \Ll(x/y)^{9\delta-1} \sum_{\ell} 1, \end{equation}
where the remaining sum is on $\ell$ satisfying \eqref{eq:ellcond}. There are $\ll \log{x} = \Ll(x/y)^{o(1)}$ such values of $\ell$, and so the right-hand side of \eqref{eq:finallydone} is at most $2^{1-j} x \Ll(x/y)^{10\delta-1}$. It remains to sum on $k$ and $j$. For each $j$, there are (crudely) $\ll \log{x} = \Ll(x/y)^{o(1)}$ values of $k$. Summing finally on $j$, we conclude that the number of $n$ counted in this piece of the argument is at most $x\Ll(x/y)^{11\delta-1}$. Earlier, we discarded at most $x\Ll(x/y)^{-1}$ integers $n$, and so 
\[ L^{\ast}(x,y)\le x\Ll(x/y)^{11\delta-1} + x\Ll(x/y)^{-1} \le x\Ll(x/y)^{12\delta-1},\] which is what we claimed above.
\end{proof}

\section{Application to multiplicative orders: Proof of Corollary \ref{cor:orders}}\label{sec:corollary}
We will need the following clean lower bound on $l(n)$, in terms of $\lambda(n)$ and the orders $l(p)$ for primes $p$ dividing $n$.

\begin{lem}\label{lem:KR} For all odd natural numbers $n$,
\[ l(n) \ge \frac{\lambda(n)}{n} \prod_{p\mid n} l(p).\]
\end{lem}

Lemma \ref{lem:KR} is due essentially to Kurlberg and Rudnick \cite{KR01}*{\S\S5.1--5.2}; see \cite{KP05}*{Lemma 5} for an explicit statement and a slightly shorter proof.

We also need the following (well-known) lemma. 
\begin{lem}\label{lem:convergence} Let $\delta \in (0,1/2)$, and let $\Pp_{\delta}= \{\text{odd primes }p: l(p) < p^{\frac{1}{2}-\delta}\}$.  Then
\[ \sum_{p \in \Pp_{\delta}} \frac{1}{p^{1-\delta}} < \infty. \]
\end{lem}
\begin{proof} For each $T\ge 1$, there are at most $T^{1-2\delta}$ primes $p \in \Pp_{\delta}\cap[1,T]$. Indeed, any such $p$ has $l(p) \le T^{\frac{1}{2}-\delta}$, so that $p \mid 2^n-1$ for some $n \le T^{\frac{1}{2}-\delta}$. But $2^n-1$ always has fewer than $n$ distinct prime divisors. Thus, the number of elements of $\Pp_{\delta}\cap[1,T]$ is at most $\sum_{n \le T^{\frac{1}{2}-\delta}}n \le T^{1-2\delta}$. Now apply partial summation.
\end{proof}

\begin{proof}[Proof of Corollary \ref{cor:orders}] We may (and will) assume that $0 < \delta < \frac12$ and that $0 \le \beta < \frac{1}{2}-\delta$. Below, we let
\[ \Mm = \Mm(x) := \Ll(x)^{\log_3{x}}, \]
noting that $\Mm$ tends to infinity faster than any power of $\Ll(x)$ while $\Mm = x^{o(1)}$ as $x\to\infty$.

Here is the plan of attack. We show that if $n \le x$ is odd, with $l(n) \le x^{\beta}$, then at least one of the following four conditions holds:
\begin{enumerate}
\item[(a)] $n \le x/\Mm$,
\item[(b)] $n$ has a squarefull divisor exceeding $\Mm$,
\item[(c)] $\lambda(n) \le x^{\frac{1}{2}+\beta + \frac{1}{2}\delta}$,
\item[(d)] $n$ has a divisor exceeding $\Mm$ composed entirely of primes $p \in \Pp_{\delta/4}$ (in the notation of Lemma \ref{lem:convergence}).
\end{enumerate}
On the other hand, we prove that the number of $n\le x$ for which at least one of (a)--(d) holds is bounded above by $x/\Ll(x)^{\frac{1}{2}-\beta-\delta}$. Corollary \ref{cor:orders} follows.

First, suppose all of (a)--(d) fail. Write $n = n_{\Pp} n_s n'$, where $n_{\Pp}$ is the largest divisor of $n$ composed of primes from $\Pp_{\delta/4}$, and $n_s$ is the largest squarefull divisor of $n/n_{\Pp}$. Then $n'$ is squarefree, every prime dividing $n'$ lies outside $\Pp_{\delta/4}$, and
\[ n' = \frac{n}{n_\Pp n_s} > \frac{x/\Mm}{n_\Pp n_s} \ge x/\Mm^3. \]
Furthermore, by Lemma \ref{lem:KR},
\begin{align*} l(n) \ge \frac{\lambda(n)}{n} \prod_{p \mid n} l(p) \ge x^{-\frac{1}{2} + \beta + \frac{1}{2}\delta} \prod_{p \mid n'} l(p) \ge x^{-\frac{1}{2} + \beta + \frac{1}{2}\delta} n'^{\frac{1}{2}-\frac{1}{4}\delta}
\ge x^{-\frac{1}{2} + \beta + \frac{1}{2}\delta} (x/\Mm^3)^{\frac{1}{2}-\frac{1}{4}\delta}
\ge x^{\beta+\frac{1}{8}\delta}.
\end{align*}
(Here and below, we always assume $x$ is sufficiently large; the notion of ``large'' may depend on $\delta$ but is uniform for $\beta \in [0,\frac{1}{2}-\delta)$.) In particular, we do not have $l(n) \le x^{\beta}$.  

Next, we show that each of (a)--(d) holds on a set of odd $n\le x$ of size at most $\frac{1}{4} x/\Ll(x)^{\frac{1}{2}-\beta-\delta}$. Conditions (a) and (b)  hold for $O(x/\Mm)$ and $O(x/\Mm^{1/2})$ values of $n \le x$, respectively, which is more than sufficient. Condition (c) is handled directly by Theorem \ref{thm:main}. Finally, the count of $n\le x$ for which (d) holds is bounded above by
\begin{align*} x \sum_{\substack{d > \Mm \\ p \mid d \Rightarrow p \in \Pp_{\delta/4}}} \frac{1}{d} &\le x \sum_{p \mid d \Rightarrow p \in \Pp_{\delta/4}} \frac{1}{d} \left(\frac{d}{\Mm}\right)^{\frac{1}{4}\delta} \\
&= x\Mm^{-\frac{1}{4}\delta} \prod_{p \in \Pp_{\delta/4}} \left(1 + \frac{1}{p^{1-\frac{1}{4}\delta}} + \frac{1}{p^{2(1-\frac{1}{4}\delta)}}+\dots\right) \\
&\ll x \Mm^{-\frac{1}{4}\delta} \exp\left(\sum_{p \in \Pp_{\delta/4}} \frac{1}{p^{1-\frac{1}{4}\delta}} \right) \\
&\ll x \Mm^{-\frac{1}{4}\delta},
\end{align*}
using Lemma \ref{lem:convergence} (with $\delta$ replaced by $\frac{1}{4}\delta$) for the last step. Again, this upper bound is tighter than required. This completes the proof.
\end{proof}

\begin{rmk}
It seems plausible that for each fixed $\theta>0$, the set of primes
$p$ with $l(p)<p^{1-\theta}$ is power-saving sparse, in the sense that
its counting function is $O(T^{1-\eta})$ for some $\eta=\eta(\theta)>0$. (For instance, this follows from Erd\H{o}s's conjecture in \cite{erdos76} that $\#\{p: l(p)=r\} = O_{\epsilon}(r^{\epsilon})$. It also follows from the weaker conjecture of Murty and Srinivasan, in the remark following Theorem 4 of \cite{MS87}, that $\sum_{p \le T} l(p)^{-1} = O_{\epsilon}(T^{\epsilon})$.)  Under this hypothesis, small modifications of the proof of
Corollary~\ref{cor:orders} show that for each fixed $\delta>0$, and all large $x$,
\[
\#\{\text{odd } n\le x: l(n)\le x^\beta\}
\le \frac{x}{\Ll(x)^{1-\beta-\delta}}
\]
uniformly in $\beta\ge0$.
\end{rmk}

\section*{Acknowledgements} I thank Casia Siegel for useful conversations. The simple proof of Lemma \ref{lem:dB} is taken from a beautiful talk by Jared Lichtman at the 2025 INTEGERS conference. 

\section*{Declaration of generative AI and AI-assisted technologies in the manuscript preparation process}
ChatGPT 5.5 was used for editing the manuscript and for pre-submission ``refereeing.'' It also suggested the use of Rankin's trick to handle condition~(d) in the proof of Corollary~\ref{cor:orders}; this replaces the more cumbersome Hardy--Ramanujanesque approach originally envisioned (cf.\ pp.~153--154 of \cite{KP05}). All other mathematical ideas and arguments are to be credited or blamed on the human author and those who trained him.
\bibliographystyle{amsplain}
\bibliography{SL}
\end{document}